\documentclass[12pt]{amsart}

\usepackage{amssymb}
\usepackage{bm}
\usepackage{graphicx}
\usepackage{psfrag}
\usepackage{hyperref}
\hypersetup{colorlinks=true, linkcolor=blue, citecolor=red, urlcolor=wine}
\usepackage{color}
\usepackage{url}
\usepackage{algpseudocode}
\usepackage{fancyhdr}
\usepackage{xy}
\input xy
\xyoption{all}
\usepackage{stmaryrd}

\newtheorem*{maintheorem*}{Main Theorem}
\newtheorem{theorem}{Theorem}[section]
\newtheorem{prop}[theorem]{Proposition}
\newtheorem{conj}[theorem]{Conjecture}
\newtheorem{question}[theorem]{Question}
\newtheorem{lemma}[theorem]{Lemma}
\newtheorem{cor}[theorem]{Corollary}

\theoremstyle{definition}
\newtheorem{definition}[theorem]{Definition}
\newtheorem{remark}[theorem]{Remark}

\numberwithin{equation}{section}

\newcommand{\nn}{\mathbb{N}}
\newcommand{\qq}{\mathbb{Q}}
\newcommand{\rr}{\mathbb{R}}
\newcommand{\zz}{\mathbb{Z}}
\providecommand\ldb{\llbracket}
\providecommand\rdb{\rrbracket}

\newcommand{\pp}{\mathsf{p}}
\newcommand{\pf}{\mathbb{P}_{\text{fin}}}
\newcommand{\cone}{\mathsf{cone}}

\newcommand{\rank}{\mathsf{rank}}

\newcommand{\slp}{\mathsf{slope}}
\newcommand{\norm}[1]{\left\lVert#1\right\rVert}



\begin{document}
	
	\mbox{}
	\title{On the system of sets of lengths and \\ the elasticity of submonoids of a\\ finite-rank free commutative monoid}
	\author{Felix Gotti}
	\address{Department of Mathematics\\UC Berkeley\\Berkeley, CA 94720 \newline \indent
				  Department of Mathematics\\Harvard University\\Cambridge, MA 02138}
	\email{felixgotti@berkeley.edu}
	\email{felixgotti@harvard.edu}
	\date{\today}
	
	\begin{abstract}
		Let $H$ be an atomic monoid. For $x \in H$, let $\mathsf{L}(x)$ denote the set of all possible lengths of factorizations of $x$ into irreducibles. The system of sets of lengths of $H$ is the set $\mathcal{L}(H) = \{\mathsf{L}(x) \mid x \in H\}$. On the other hand, the elasticity of $x$, denoted by $\rho(x)$, is the quotient $\sup \mathsf{L}(x)/\inf \mathsf{L}(x)$ and the elasticity of $H$ is the supremum of the set $\{\rho(x) \mid x \in H\}$. The system of sets of lengths and the elasticity of $H$ both measure how far is $H$ from being half-factorial, i.e., $|\mathsf{L}(x)| = 1$ for each $x \in H$.
		
		Let $\mathcal{C}$ denote the collection comprising all submonoids of finite-rank free commutative monoids, and let $\mathcal{C}_d = \{H \in \mathcal{C} \mid \text{rank}(H) = d\}$. In this paper, we study the system of sets of lengths and the elasticity of monoids in $\mathcal{C}$. First, we construct for each $d \ge 2$ a monoid in $\mathcal{C}_d$ having extremal system of sets of lengths. It has been proved before that the system of sets of lengths does not characterize (up to isomorphism) monoids in $\mathcal{C}_1$. Here we use our construction to extend this result to $\mathcal{C}_d$ for any $d \ge 2$. On the other hand, it has been recently conjectured that the elasticity of any monoid in $\mathcal{C}$ is either rational or infinite. We conclude this paper by proving that this is indeed the case for monoids in $\mathcal{C}_2$ and for any monoid in $\mathcal{C}$ whose corresponding convex cone is polyhedral.
	\end{abstract}
\bigskip

\maketitle

\section{Introduction}
\label{sec:intro}

Many interesting integral domains fail to be unique factorization domains. The~interest in measuring such a failure dates back to the mid-nineteen century. For an integral domain $R$ and $x \in R$, let $\mathsf{Z}(x)$ be the set of all possible factorizations of $x$ into irreducibles. The domain $R$ is called \emph{half-factorial} if for all $x \in R$ any two $z,z' \in \mathsf{Z}(x)$ involve the same number of irreducibles (counting repetitions). L.~Carlitz~\cite{lC60} proved that a ring of integers is half-factorial if and only if the size of its class group is at most~$2$. The phenomenon of non-unique factorizations of many other families of integral domains has been studied since then (see~\cite{dA97,ASS94} and references therein). The study of the non-uniqueness of factorizations on commutative cancellative monoids has also earned significant attention during the last few decades (see \cite{CGP14,CGTV16,GZ18}). This is mainly because many factorization properties of an integral domain $R$ are purely multiplicative in nature and, therefore, can be understood by studying only its multiplicative monoid $R \! \setminus \! \{0\}$.

To measure how far an integral domain or a commutative cancellative monoid is from being half-factorial, many algebraic and arithmetical invariants have proved to be useful. Such invariants include the class group (of a Krull domain/monoid)~\cite{lC60}, the system of sets of lengths~\cite{aG16}, the elasticity~\cite{dA97}, and the set of distances~\cite{CGP14}.
\medskip

\noindent {\bf Notation.} For $d \ge 1$, we let $\mathcal{C}_d$ denote the collection consisting of all rank-$d$ submonoids of any finite-rank free commutative monoid. In addition, we set $\mathcal{C} := \cup_{d \ge 1} \, \mathcal{C}_d$ and $\mathcal{C}_{\ge 2} := \cup_{d \ge 2} \, \mathcal{C}_d$.
\medskip

The class $\mathcal{C}$ generalizes the class of all reduced affine monoids, i.e., monoids in $\mathcal{C}$ that are finitely generated. The interested reader may find a self-contained treatment of affine monoids in~\cite{BG09}. In this paper, we investigate the phenomenon of non-unique factorizations of monoids in $\mathcal{C}$. To understand how far from half-factorial the monoids in the class $\mathcal{C}$ can be, we will investigate their systems of sets of lengths and their elasticities.

The system of sets of lengths $\mathcal{L}(H)$ of an atomic monoid $H$ encodes significant information about the arithmetic of factorizations of $H$. This explains why the system of sets of lengths is perhaps the most investigated factorization invariant in the context of atomic monoids. In particular, the search for classes of atomic monoids having extremal systems of sets of lengths has been frequently explored in the recent literature (see~\cite{fK99} and~\cite{FNR17}). In the first part of this paper we exhibit, for every $d \ge 2$, a monoid $H_d \in \mathcal{C}_d$ having full system of sets of lengths.

In the 1970s, W. Narkiewicz posed the question of whether the arithmetic describing the non-uniqueness of factorizations in a Krull domain could be used to characterize its class group (for affirmative answers to this, see~\cite[Sections~7.1 and~7.2]{GH06b}). In general, the question of whether $\mathcal{L}(H)$ completely determines a monoid $H$ (up to isomorphism) inside a distinguished family of atomic monoids has been previously studied (see~\cite{ACHP07}, \cite[Section~4]{fG17c}, and~\cite[Section~6]{aG16}). In the context of Krull monoids, this question is known as the Characterization Problem, which is still open and being actively investigated. Similar questions have been answered for numerical monoids~\cite{ACHP07} and Puiseux monoids~\cite{fG17c}. Here, we argue that, for any $d \ge 2$, the system of sets of lengths does not characterize (up to isomorphism) the monoids in the class $\mathcal{C}_d$.

The elasticity was first used in~\cite{rV90} as a tool to measure the phenomenon of non-unique factorizations in rings of integers of algebraic number fields. The elasticity $\rho(H) \in \rr_{\ge 1} \cup \{\infty\}$ of an atomic monoid $H$ also measures how far is $H$ from being half-factorial; in particular, $H$ is half-factorial if and only if $\rho(H) = 1$. Although the elasticity encodes substantially less amount of information than the system of sets of lengths does, the former is, in general, much easier to compute. The elasticity of integral domains and atomic monoids has been considered by many authors (see, for instance, \cite{dA97}, \cite{BOP17}, \cite{CGGR01}, and \cite{GO17}). In the last section of this paper, we turn to study the elasticity of monoids in $\mathcal{C}$. In particular, we prove that the elasticity of any monoid in~$\mathcal{C}_2$ is either rational or infinity. Finally, we show that if the convex cone of $H \in \mathcal{C}$ is a polyhedral cone, then $\rho(H)$ is also rational or infinite.
\bigskip

\section{Background on Monoids and Factorization Theory}
\label{sec:background}

In this section we introduce most of the relevant concepts concerning commutative monoids and factorization theory required to follow our exposition. For deeper background or undefined terms on these subjects, the reader may want to consult A.~Geroldinger and F.~Halter-Koch~\cite{GH06b} and P.~Grillet~\cite{pG01}.
\medskip

\noindent {\bf General Notation.} Recall that $\nn := \{0,1,2,\dots\}$. If $a,b \in \zz$ and $a \le b$, then we let the interval $\ldb a,b \rdb$ denote the set of integers $\{z \in \zz \mid a \le z \le b\}$. In addition, for $X \subseteq \rr$ and $r \in \rr$, we set
\[
	X_{\ge r} := \{x \in X \mid x \ge r\}; 
\]
in a similar manner, we use the symbol $X_{> r}$. Also, if $Y \subseteq \rr^d$ for some $d \in \nn \setminus \{0\}$, then we set $Y^\bullet := Y \setminus \{0\}$.
Finally, we introduce the nonstandard notation
\[
	\pf := \big\{ \{0\},\{1\} \big\} \cup \big\{S \subset \zz_{\ge 2} \mid S \ \text{ is finite} \big\}
\]
because the set $\pf$ will play an important role in Section~\ref{sec:sets of lengths}.
\medskip

A \emph{monoid} is commonly defined in the literature as a semigroup along with an identity element. However, in the following all monoids are commutative and cancellative, and we omit to mention these two attributes accordingly. As we only consider commutative monoids, unless otherwise specified we will use additive notation. In particular, the identity element of a monoid $H$ is denoted by $0$, and we let $H^\bullet$ denote the set $H \! \setminus \! \{0\}$. For $x,y \in H$, we say that $y$ \emph{divides} $x$ \emph{in} $H$ and write $y \mid_H x$ provided that $x = y + z$ for some $z \in H$. We write $H = \langle S \rangle$ when $H$ is generated as a monoid by a set $S$. If $H$ can be generated as a monoid by a finite set, we say that $H$ is \emph{finitely generated}.

Every monoid $H$ we consider here is assumed to be \emph{reduced}, which means that the only invertible element of $H$ is $0$. An element $a \in H^\bullet$ is called an \emph{atom} if for each pair of elements $y,z \in H$ such that $a = y+z$ either $y = 0$ or $z = 0$. The set consisting of all atoms of $H$ is denoted by $\mathcal{A}(H)$, that is,
\[
	\mathcal{A}(H) := H^\bullet \setminus \big( H^\bullet + H^\bullet \big).
\]
Since $H$ is reduced, it follows that $\mathcal{A}(H)$ will be contained in each generating set of~$H$. If $\mathcal{A}(H)$ generates $H$, then $H$ is said to be \emph{atomic}. All monoids addressed in this paper are atomic.

For any monoid $H$ there exist an abelian group $\text{gp}(H)$ and a monoid homomorphism $\iota \colon H \hookrightarrow \text{gp}(H)$ such that any monoid homomorphism $\phi \colon H \to G$ (where $G$ is a group) uniquely factors through $\iota$. The group $\text{gp}(H)$, which is unique up to isomorphism, is called the \emph{difference group} (or \emph{Grothendieck group}) of $H$. If $H$ is a monoid in $\mathcal{C}$, then the \emph{rank} of $H$, denoted by $\rank(H)$, is the rank of the abelian group $\text{gp}(H)$, that is, the dimension of the $\qq$-space $\qq \otimes_\zz \text{gp}(H)$. The monoid $H$ is \emph{torsion-free} if $nx = ny$ for some $n \in \nn$ and $x,y \in H$ implies that $x = y$. A monoid is torsion-free if and only if its difference group is torsion-free (see \cite[Section~2.A]{BG09}).

A multiplicative commutative monoid $F$ is \emph{free on} a subset $A$ of $F$ if every element $x \in F$ can be written uniquely in the form
\[
	x = \prod_{a \in A} a^{\mathsf{v}_a(x)},
\]
where $\mathsf{v}_a(x) \in \nn$ and $\mathsf{v}_a(x) > 0$ only for finitely many $a \in A$. It is well known that for each set $A$, there exists a unique free commutative monoid on $A$ (up to isomorphism). The free commutative monoid on $\mathcal{A}(H)$, denoted by $\mathsf{Z}(H)$, is called the \emph{factorization monoid} of $H$, and the elements of $\mathsf{Z}(H)$ are called \emph{factorizations}. If $z = a_1 \dots a_n$ is a factorization in $\mathsf{Z}(H)$ for some $n \in \nn$ and $a_1, \dots, a_n \in \mathcal{A}(H)$, then $n$ is called the \emph{length} of $z$ and is denoted by $|z|$. The unique monoid homomorphism $\phi \colon \mathsf{Z}(H) \to H$ satisfying $\phi(a) = a$ for all $a \in \mathcal{A}(H)$ is called the \emph{factorization homomorphism} of $H$, and for each $x \in H$ the set
\[
	\mathsf{Z}(x) := \mathsf{Z}_H(x) := \phi^{-1}(x) \subseteq \mathsf{Z}(H)
\]
is called the \emph{set of factorizations} of $x$. Observe that $H$ is atomic if and only if $\mathsf{Z}(x)$ is nonempty for all $x \in H$ (notice that $\mathsf{Z}(0) = \{\emptyset\}$). The monoid $H$ is called a \emph{finite factorization monoid} or, simply, an \emph{FF-monoid} provided that $|\mathsf{Z}(x)| < \infty$ for all $x \in H$. For each $x \in H$, the \emph{set of lengths} of $x$ is defined by
\[
	\mathsf{L}(x) := \mathsf{L}_H(x) := \{|z| \ | \ z \in \mathsf{Z}(x)\}.
\]
If $|\mathsf{L}(x)| < \infty$ for all $x \in H$, then $H$ is called a \emph{bounded factorization monoid} or, for short, a \emph{BF-monoid}. Clearly, if a monoid is an FF-monoid, then it is also a BF-monoid. The \emph{system of sets of lengths} of $H$ is defined by
\[
	\mathcal{L}(H) := \{\mathsf{L}(x) \mid x \in H\}.
\]
The structure of the system of sets of lengths of Krull monoids was first studied by Geroldinger in~\cite{aG88}. Since then the structure of the system of sets of lengths of many other classes of monoids and domains has been investigated; see reference in~\cite{aG16}, which is a survey on sets of lengths and the role they play in factorization theory. We say that a BF-monoid $H$ has \emph{full system of sets of lengths} if $\mathcal{L}(H) = \pf$. Note that $\pf$ is the largest (under inclusion) system of sets of lengths a BF-monoid can have.

An important factorization statistic related with the sets of lengths of an atomic monoid $H$ is the elasticity. The \emph{elasticity} $\rho(x)$ of an element $x \in H^\bullet$ is defined as
\[
	\rho(x) := \rho_H(x) := \frac{\sup \mathsf{L}(x)}{\inf \mathsf{L}(x)}.
\]
Note that $\rho(x) \in \qq_{\ge 1} \cup \{\infty\}$ for all $x \in H^\bullet$. On the other hand, the \emph{elasticity} of $H$ is defined to be
\[
	\rho(H) := \sup \{\rho(x) \mid x \in H\}.
\]
The \emph{set of elasticities} of $H$ is $\mathcal{R}(H) := \{\rho(x) \mid x \in H^\bullet\}$. We say that $H$ is \emph{fully elastic} provided that $\mathcal{R}(H) = \{q \in \qq \mid 1 \le q \le \rho(H) \}$. The concept of elasticity was introduced back in 1980 by R.~Valenza~\cite{rV90} in the context of algebraic number theory\footnote{Valenza's paper appeared in 1990; however, it was actually submitted 10 years earlier.}. The system of sets of lengths and the elasticity have received a great deal of attention in the literature in recent years (see, for instance,~\cite{ACHP07,CGGR01,FG08,GO17}). 

A very special family of atomic monoids is that of all \emph{numerical monoids}, i.e., cofinite submonoids of $\nn$. Each numerical monoid $H$ has a unique minimal generating set, which is finite; such a unique minimal generating set is precisely $\mathcal{A}(H)$. As a result, every numerical monoid is atomic and contains only finitely many atoms. The reader can find an introduction to numerical monoids in~\cite{GR09}. We end this section with the following realization theorem of A. Geroldinger and W. Schmid, which will be crucial in the proof of Theorem~\ref{thm:GAM with full system of sets of lengths}.

\begin{theorem} \cite[Theorem~3.3]{GS17} \label{thm:realization theorem in numerical monoids}
	Let $L \subset \zz_{\ge 2}$ be a finite nonempty set, and let $f \colon L \to \zz_{\ge 1}$ be a map. Then there exist a numerical monoid $H$ and a squarefree element $x \in H$ such that
	\[
		\mathsf{L}(x) = L \ \text{and} \ |\mathsf{Z}_k(x)| = f(k) \ \text{for every} \ k \in L,
	\]
	where $\mathsf{Z}_k(x) := \{z \in \mathsf{Z}(x) \mid |z| = k\}$.
\end{theorem}
\bigskip

\section{Background on Convex Cones}
\label{sec:cones}

The geometry needed in this paper is basic and takes place in either the $\rr$-space $\rr^d$ or the $\qq$-space $\qq^d$, mostly the latter one. We denote the standard inner product of $\rr^d$ by $\langle \, , \rangle$, that is, for all $x = (x_1, \dots, x_d)$ and $y = (y_1, \dots, y_d)$ in $\rr^d$,
\[
	\langle x, y \rangle = \sum_{i=1}^d x_i y_i.
\]
As usual, for $x \in \rr^d$ we let $\norm{x}$ denote the Euclidean norm of $x$. We always consider the space $\rr^d$ endowed with the topology induced by the Euclidean norm. Finally, we let the $\qq$-space $\qq^d$ inherit the inner product and the topology of $\rr^d$.

Let $V$ be a vector space over an ordered field. A nonempty subset $C$ of $V$ is called a \emph{convex cone} provided that $C$ is closed under linear combinations with nonnegative coefficients. Convex cones are clearly convex sets containing $0$. If $X$ is a nonempty subset of $V$, then the set
\[
	\cone(X) := \big\{ c_1 x_1 + \dots + c_n x_n \mid x_i \in X \ \text{and} \ c_i \ge 0 \ \text{for each} \ i \in \ldb 1,n \rdb \big\}
\]
is the smallest cone containing $X$. In this case, $\cone(X)$ is called the \emph{cone generated by} $X$. We say that a cone $C$ is \emph{pointed} if $C \cap -C = \{0\}$. Unless otherwise stated, we assume that the cones we consider here are pointed.

For a nonzero vector $u \in \rr^d$, consider the hyperplane $\sigma := \{x \in \rr^d \mid \langle x, u \rangle = 0 \}$, and denote the closed half-spaces $\{x \in \rr^d \mid \langle x, u \rangle \le 0 \}$ and $\{x \in \rr^d \mid \langle x, u \rangle \ge 0 \}$ by $\sigma^-$ and $\sigma^+$, respectively. If a cone $C$ satisfies that $C \subseteq \sigma^-$ (resp., $C \subseteq \sigma^+$), then $\sigma$ is called a \emph{supporting hyperplane} of $C$ and $\sigma^-$ (resp., $\sigma^+$) is called a \emph{supporting half-space} of $C$. A subset $F$ of $C$ is a \emph{face} if there exists a supporting hyperplane $\sigma$ of $C$ such that $F = C \cap \sigma$. The cone $C$ is said to be \emph{polyhedral} provided that it has only finitely many faces. The Farkas-Minkowski-Weyl Theorem states that a convex cone is polyhedral if and only if it is finitely generated.

A pointed cone in $\rr^d$ is called \emph{rational} provided that it can be generated by finitely many rational (or, equivalently, integer) vectors. Clearly, every rational cone is polyhedral. On the other hand, a \emph{lattice} in $\rr^d$ is an additive subgroup of $\rr^d$ generated by $\rr$-linearly independent vectors.

\begin{theorem}[Gordan's Lemma]
	Let $C$ be a rational cone in $\rr^d$, and let $L \subset \qq^d$ be a lattice. Then $C \cap L$ is a finitely generated monoid.
\end{theorem}

Let $H$ be a monoid. A submonoid $S$ of $H$ is called \emph{divisor-closed} if for all $x \in H$ and $s \in S$ the fact that $x \mid_H s$ implies that $x \in S$. The monoid $H$ is called \emph{primary} if it is nontrivial and its only divisor-closed submonoids are $\{0\}$ and $H$. Numerical monoids and, in general, additive submonoids of $\qq_{\ge 0}$ are examples of primary monoids. For information about primary monoids, see \cite[Section~2.7]{GH06b}.
Primary monoids in $\mathcal{C}$ have the following geometric characterization.

\begin{theorem} \cite[Theorem~2.4]{GHL95} \label{thm:primary geometric characterization}
	For any $d \ge 2$, a monoid $H$ in $\mathcal{C}_d$ is primary if and only if $\cone(H)^\bullet$ is an open subset of $\qq \otimes_\zz \emph{gp}(H) \subseteq \qq^d$.
\end{theorem}
\bigskip

\section{The System of Sets of Lengths}
\label{sec:sets of lengths}

In this section we construct, for each $d \ge 2$, a monoid $H$ in $\mathcal{C}_d$ having full system of sets of lengths, that is $\mathcal{L}(H) = \pf$. To begin with, let us argue the following lemma.

\begin{lemma} \label{lem:basic isomorphic representation}
	For every $d \in \zz_{\ge 1}$, each monoid in $\mathcal{C}_d$ is isomorphic to an additive submonoid of $\nn^d$ of rank $d$.
\end{lemma}

\begin{proof}
	Let $H$ be a monoid in $\mathcal{C}_d$, and suppose that $H$ is a submonoid of a free commutative monoid of rank $r$ for some $r \in \nn$ with $r \ge d$. There is no loss of generality in assuming that $H$ is a submonoid of $\nn^r \subseteq \qq^r$. Let $V$ be the subspace of the $\qq$-sapce $\qq^r$ generated by~$H$. Since $H$ has rank $d$, the subspace $V$ has dimension $d$. Now consider the submonoid $H' := \nn^r \cap V$ of $\nn^r$. As $H'$ is the intersection of the rational cone $\cone(\nn^r \cap V)$ and the lattice $\zz^r \cap V \cong \zz^d$, it follows by Gordan's Lemma that $H'$ is finitely generated. On the other hand, $H \subseteq H' \subseteq V$ guarantees that $\rank(H') = d$. Since $H'$ is a finitely generated additive submonoid of $\nn^r$ of rank $d$, it follows by~\cite[Proposition~2.17]{BG09} that $H'$ is isomorphic to an additive submonoid of $\nn^d$. This, in turn, implies that $H$ is isomorphic to an additive submonoid of $\nn^d$.
\end{proof}

We proceed to show that any monoid in $\mathcal{C}$ is an FF-monoid and, therefore, a BF-monoid.

\begin{prop}
	Each monoid in $\mathcal{C}$ is an FF-monoid.
\end{prop}

\begin{proof}
	By Lemma~\ref{lem:basic isomorphic representation}, it suffices to show that for every $d \in \zz_{\ge 1}$, any additive submonoid $H$ of $\nn^d$ is an FF-monoid. Fix $x \in H$. It is clear that $\langle x, y \rangle \ge 0$ for all $y \in H$. Thus, $y \mid_H x$ implies that $\norm{y} \le \norm{x}$. As a result, the set $\{a \in \mathcal{A}(H) \mid a |_H x\}$ is finite, which implies that $\mathsf{Z}(x)$ is also finite. Hence $H$ is an FF-monoid, as desired.
\end{proof}

In particular, every monoid in $\mathcal{C}$ is a BF-monoid. Therefore to show that a monoid~$H$ in $\mathcal{C}$ has full system of sets of lengths, it suffices to verify that $\pf \subseteq \mathcal{L}(H)$. Before proceeding with our main result, let us exhibit some examples of families of atomic monoids and domains that have recently been proved to have full systems of sets of lengths.
\medskip

The first family of atomic monoids with full systems of sets of lengths was given by F.~Kainrath~\cite{fK99} in the context of Krull monoids. A monoid $K$ is called a \emph{Krull monoid} if there exists a monoid homomorphism $\phi \colon K \to D$, where $D$ is a free commutative monoid, satisfying the next two conditions:
\begin{enumerate}
	\item if $a, b \in K$ and $\phi(a) \mid_D \phi(b)$, then $a \mid_K b$;
	\vspace{2pt}
	\item for every $d \in D$ there exist $a_1, \dots, a_n \in K$ with $d = \gcd\{\phi(a_1), \dots, \phi(a_n)\}$.
\end{enumerate}
The basis elements of $D$ are called the \emph{prime divisors} of $K$, and the abelian group $\text{Cl}(K) := D/\phi(K)$ is called the \emph{class group} of $K$. As Krull monoids are isomorphic to submonoids of free commutative monoids, Krull monoids are atomic (see \cite[Section~2.3]{GH06b} for further details about Krull monoids).

\begin{theorem} \cite[Theorem~1]{fK99}
	Let $H$ be a Krull monoid with infinite class group in which every divisor class contains a prime divisor. For a finite subset $L$ of $\zz_{\ge 2}$ there exists $x \in H$ such that $\mathsf{L}(x) = L$.
\end{theorem}

In the same direction, S.~Frisch has proved that the multiplicative monoid of the domain of integer-valued polynomials $\text{Int}(\zz)$ also has full system of sets of lengths (see ~\cite{sF13}). This result was recently generalized in~\cite{FNR17} to the domain $\text{Int}(\mathcal{O}_K)$ of polynomials over a given number field $K$ stabilizing the ring of integers $\mathcal{O}_K$.

\begin{theorem} \cite[Theorem~1]{FNR17}
	Let $K$ be a number field with ring of integers $\mathcal{O}_K$. Moreover, let $1 \le m_1 \le \dots \le m_n$ be natural numbers. Then there exists a polynomial in $\emph{Int}(\mathcal{O}_K)$ with $n$ essentially different factorizations into irreducible polynomials in $\emph{Int}(\mathcal{O}_K)$ where the lengths of these factorizations are $m_1 + 1, \dots, m_n + 1$.
\end{theorem}

A \emph{Puiseux monoid} is an additive submonoid of $\qq_{\ge 0}$. Although Puiseux monoids are natural generalizations of numerical monoids, the former are not necessarily finitely generated or atomic. Moreover, if an atomic Puiseux monoid $H$ is not isomorphic to a numerical monoid, then $|\mathcal{A}(H)| = \infty$. The atomic structure and factorization theory of Puiseux monoids have only been studied recently (see \cite{fG17} and~\cite{GG17}). There are Puiseux monoids having full systems of sets of lengths.

\begin{theorem} \cite[Theorem~3.6]{fG17c}
	There exists a Puiseux monoid with full system of sets of lengths.
\end{theorem}

It was proved in~\cite{GHL07} that for $d$ large enough there exists a primary monoid in $\mathcal{C}_d$ having full system of sets of lengths. Now we exhibit a primary monoid in $\mathcal{C}_2$ with full system of sets of lengths. Then we use such a monoid to construct, for every $d \in \zz_{\ge 2}$, a monoid in $\mathcal{C}_d$ with full system of sets of lengths. For the remaining of this paper, we use the following notation: for a nonzero $x \in \rr^2_{\ge 0}$, we let $\slp(x) \in \rr_{\ge 0} \cup \{\infty\}$ denote the slope of the line $\rr x$, and for $X \subset \rr_{\ge 0}^2$ we set
\[
	\slp(X) := \{\slp(x) \mid x \in X^\bullet\}.
\]

\begin{theorem} \label{thm:GAM with full system of sets of lengths}
	There exists a primary monoid in $\mathcal{C}_2$ having full system of sets of lengths.
\end{theorem}

\begin{proof}
	As $\pf$ is a countable collection, we can list its members. Let $S_1, S_2, \dots$ be an enumeration of the members of $\pf$. Fix $\ell,L \in \qq_{> 0}$ such that $\ell < L$. Now take a sequence $\{a_n\}$ of elements in $\nn^2$ such that the sequence $\{\slp(a_{2n-1})\}$ strictly decreases to $\ell$ and the sequence $\{\slp(a_{2n})\}$ strictly increases to $L$. In addition, assume that
	\[
		\max\{\slp(a_{2n-1}) \mid n \in \nn \} < \min \{\slp(a_{2n}) \mid n \in \nn\}.
	\]
	Now for every $n \in \nn$, we use Theorem~\ref{thm:realization theorem in numerical monoids} to obtain an additive submonoid $H_n$ of $\nn a_n$ and an element $x_n \in H_n$ such that $\mathsf{L}_{H_n}(x_n) = S_n$ (note that $\nn a_n$ contains an isomorphic copy of every numerical monoid). After rescaling each $H_1, H_2, \dots$ (in this order) one can guarantee that
	\begin{equation} \label{eq:rescaling consequence 2}
		\min \{ \norm{a} \mid a \in \mathcal{A}(H_{n+1}) \} > \max \big\{\norm{x_n}, \max \{\norm{a} \mid a \in \mathcal{A}(H_n)\} \big\}
	\end{equation}
	for every $n \in \nn$. Take $H$ to be the smallest additive submonoid of $\nn^2$ containing every monoid $H_n$. Clearly, $H$ is generated by the set $A := \cup_{n \in \nn} \mathcal{A}(H_n)$. Let us prove that each $S_n$ is contained in $\mathcal{L}(H)$.
	\vspace{4pt}
	
	\noindent {\it Claim:} If $a \in \mathcal{A}(H)$ divides $x_n$ in $H$, then $\slp(a) = \slp(x_n)$.
	\vspace{3pt}
	
	\noindent {\it Proof of Claim:} Suppose, by way of contradiction, that there exist $n \in \nn$ and $a \in \mathcal{A}(H)$ such that $a \mid_H x_{n}$ and $\slp(a) \neq \slp(x_n)$. Assume first that $n$ is even, namely, $n = 2k$. Since $a \mid_H x_{2k}$, there exists $b \in H$ such that $a + b = x_{2k}$. Because $a$ and~$b$ are vectors located in the interior of the first quadrant, $x_{2k}$ is the longest diagonal of the lattice parallelogram determined by the vectors $a$ and~$b$. Hence $\norm{a} < \norm{x_{2k}}$. Observe that if $\slp(a) < \slp(x_{2k})$, then we can take $a' \in \mathcal{A}(H)$ satisfying that $a' \mid_H x_{2k}$ and $\slp(a') > \slp(x_{2k})$. Then we can assume without loss of generality that $\slp(a) > \slp(x_{2k})$. This, along with the fact that $x_{2k}$ has an even index, ensures that $a \in \mathcal{A}(H_m)$ for some index $m > 2k$. Now the inequality~({\ref{eq:rescaling consequence 2}}) guarantees that $\norm{a} > \norm{x_{2k}}$, which contradicts the already-established inequality $\norm{a} < \norm{x_{2k}}$. Hence every atom $a$ of $H$ dividing $x_{2k}$ in $H$ must satisfy that $\slp(a) = \slp(x_{2k})$. The case when $n$ is odd can be argued similarly. Thus, the claim follows.
	
	As a direct consequence of the above claim, $\mathsf{L}_{H}(x_n) = \mathsf{L}_{H_n}(x_n) = S_n$ for every $n \in \nn$. Hence $\pf \subseteq \mathcal{L}(H)$, and so $H$ has full system of sets of lengths. Finally, notice that
	\[
		\cone(H)^\bullet := \{x \in \qq^2 \mid x \neq 0 \ \text{ and } \ \ell < \slp(x) < L\},
	\]
	which is an open subset of $\qq^2$. Thus, it follows by Theorem~\ref{thm:primary geometric characterization} that $H$ is primary, which concludes the proof.
\end{proof}

The reader might have noticed that the argument we presented in the proof of Theorem~\ref{thm:GAM with full system of sets of lengths} can be simplified by using only one limit slope instead of two. We record this parallel result in the next proposition for future reference. However, it is not hard to see that the monoid resulting from using only one slope does not have the extra desirable property of being primary.

\begin{prop} \label{prop:a GAM with full system of sets of lengths}
	There exists a monoid $H$ in $\mathcal{C}_2$ having full system of sets of lengths such that $\slp(H)$ has only one limit point.
\end{prop}

\begin{proof}
	It is left to the reader as it follows the same argument as the proof of Theorem~\ref{thm:GAM with full system of sets of lengths}.
\end{proof}

As every submonoid of $\nn$ is isomorphic to a numerical monoid, and the elasticity of a numerical monoid is finite \cite[Theorem 2.1]{CHM06}, no monoid in $\mathcal{C}_1$ can have full system of sets of lengths. However, Theorem~\ref{thm:GAM with full system of sets of lengths} can be used to construct, for each $d \ge 2$, a maximal-rank submonoid of $\nn^d$ having full system of sets of lengths.

\begin{cor} \label{cor:GAMS with full system of sets of lengths}
	For every $d \ge 2$, there exists a monoid in $\mathcal{C}_d$ having full system of sets of lengths.
\end{cor}

\begin{proof}
	The case $d=2$ is Theorem~\ref{thm:GAM with full system of sets of lengths}. Suppose, therefore, that $d \ge 3$. By Theorem~\ref{thm:GAM with full system of sets of lengths} there exists a submonoid $H'$ of $\nn^d$ with $\rank(H') = 2$ such that $p_i(H') = \{0\}$ for every $i \in \ldb 3,d \rdb$. Take vectors $v_3, \dots, v_d \in \nn^d$ such that the rank of the submonoid
	\[
		H := \langle H' \cup \{v_3, \dots, v_d\} \rangle
	\]
	of $\nn^d$ is $d$. Since $v_i \notin H'$ for each $i \in \ldb 3,d \rdb$, it follows that $H'$ is a divisor-closed submonoid of $H$. Therefore $\mathsf{L}_{H'}(x) = \mathsf{L}_H(x)$ for all $x \in H'$. As a result, $\mathcal{L}(H') = \pf$ implies that $\mathcal{L}(H) = \pf$. Thus, $H$ has full system of sets of lengths.
\end{proof}
\medskip

We would like to remark that the submonoid $H$ of $\nn^d$ (for $d \ge 3$) constructed in Corollary~\ref{cor:GAMS with full system of sets of lengths} is not primary. Notice, for instance, that $H \cap \{x \in \qq^d \mid p_d(x) = 0 \}$ is a nonempty proper divisor-closed submonoid of $H$. However, the reader is invited to prove the following conjecture, which we believe to be true.

\begin{conj}
	For every dimension $d \ge 3$, there exists a primary monoid in $\mathcal{C}_d$ having full system of sets of lengths.
\end{conj}
\medskip

We conclude this section answering the characterization problem for sets of lengths in each class $\mathcal{C}_d$. Let $\phi \colon H \to H'$ be a monoid isomorphism, where $H$ and $H'$ are monoids in $\mathcal{C}$. Then $\phi$ extends to a group isomorphism $\text{gp}(H) \to \text{gp}(H')$. In particular, if two monoids in $\mathcal{C}$ are isomorphic, they have the same rank. Therefore Corollary~\ref{cor:GAMS with full system of sets of lengths} immediately implies that the system of sets of lengths does not characterize monoids in $\mathcal{C}_{\ge 2}$. On the other hand, it was proved in~\cite{ACHP07} that the system of sets of lengths does not characterize monoids in $\mathcal{C}_1$. Now we extend these two observations.

\begin{prop}
	For any $d \ge 2$, the system of sets of lengths does not characterize monoids inside the class $\mathcal{C}_d$.
\end{prop}

\begin{proof}
	First, suppose that $d=2$. Take $H, H' \in \mathcal{C}_d$, and let $\phi \colon H \to H'$ be a monoid isomorphism. Then $\phi$ extends to a group isomorphism $\text{gp}(H) \to \text{gp}(H')$ and, since $\qq$ is a flat $\zz$-module, $\phi$ also extends to an isomorphism
	\[
		\bar{\phi} \colon \qq \otimes_\zz \text{gp}(H) \to \qq \otimes_\zz \text{gp}(H')
	\]
	of $\qq$-spaces. As $\bar{\phi}$ is a linear transformation, it must be continuous. Thus, if the monoids $H$ and $H'$ are isomorphic, then $\slp(\mathcal{A}(H))$ and $\slp(\mathcal{A}(H'))$ have the same number of limit points. Now if $H$ and $H'$ are the monoids in $\mathcal{C}_2$ constructed in Proposition~\ref{prop:a GAM with full system of sets of lengths} and in the proof of Theorem~\ref{thm:GAM with full system of sets of lengths}, respectively, then $\slp(\mathcal{A}(H)) \subset \rr$ has one limit point and $\slp(\mathcal{A}(H')) \subset \rr$ has two limit points. Hence $\mathcal{L}(H) = \pf = \mathcal{L}(H')$, but $H$ and $H'$ are not isomorphic.
	
	Suppose, on the other hand, that $d > 2$. Notice that the monoid $H$ of $\mathcal{C}_d$ constructed in Corollary~\ref{cor:GAMS with full system of sets of lengths} satisfies that $\cone(H)$ is polyhedral. This is because there exists one supporting plane of $\cone(H)$ containing all but finitely many elements of $\mathcal{A}(H)$. Slightly modifying the proof of Corollary~\ref{cor:GAMS with full system of sets of lengths}, we can construct a monoid $H'$ in $\mathcal{C}_d$ with one of its one-dimensional extreme rays containing two atoms. As an isomorphism $\phi \colon H \to H'$ would send atoms to atoms and its $\qq$-linear extension $\bar{\phi} \colon \qq \otimes_\zz \text{gp}(H) \to \qq \otimes_\zz \text{gp}(H')$ would send one-dimensional faces of $\cone(H)$ to one-dimensional faces of $\cone(H')$, such isomorphism $\phi$ cannot exist. Hence $H$ and $H'$ are not isomorphic monoids even though $\mathcal{L}(H) = \pf = \mathcal{L}(H')$.
\end{proof}
\bigskip

\section{Rationality of the Elasticity}
\label{sec:BF monoid with full elasticity}

We now turn our attention to the elasticity of monoids in $\mathcal{C}$. The following question was asked by G. Lettl and S. Tringali, and then it was posed in~\cite{sT17}.

\begin{question} \label{quest:rational elasticity}
	 Is always the elasticity of a submonoid of a free commutative monoid of finite rank either rational or infinite?
\end{question}
	
	Clearly, the submonoids of free commutative monoids of finite rank are precisely those in $\mathcal{C}$. In Theorem~\ref{thm:elasticity in dimension two} and Theorem~\ref{thm:elasticity of polyhedral GAMS}, we shall provide two positive partial answers to Question~\ref{quest:rational elasticity}.
	
	Before delving into the actual question, let us extend the notion of the set of lengths and the elasticity for submonoids $H$ of $\nn$ so they are both defined in terms of any fixed finite generating set of $H$ (not necessarily $\mathcal{A}(H)$). Similar generalizations of other arithmetic invariants have been useful in the past to study aspects of the non-unique factorization theory of certain classes of monoids, including arithmetical congruence monoids~\cite{BCS08}. In particular, the generalized set of lengths was first used in the context of numerical monoids in~\cite{CDHK10}, where a similar relaxation of the set of distances was studied.

\begin{definition}
	Let $k \in \nn$ and $n_1, \dots, n_k \in \nn^\bullet$. Then for any nonzero $x \in \langle n_1, \dots, n_k \rangle$, we define the \emph{generalized set of lengths} of $x$ with respect to the distinguished generators $n_1, \dots, n_k$ to be
	\[
		\mathsf{L}_g(x) = \bigg\{c_1 + \dots + c_k \ \bigg{|} \ c_1, \dots, c_k \in \nn \ \text{ and } \ \sum_{i=1}^k c_i n_i = x \bigg\} \subset \nn^\bullet.
	\]
	Similarly, the \emph{generalized elasticity} of $x$ with respect to $n_1, \dots, n_k$ is defined to be
	\[
		\rho_g(x) = \frac{\max \mathsf{L}_g(x)}{\min \mathsf{L}_g(x)}.
	\]
\end{definition}

\begin{lemma} \label{lem:set of pseudo-elasticity}
	For $k \in \zz_{\ge 2}$, take $n_1, \dots, n_k \in \nn$ such that $n_1 < \dots < n_k$. If the set
	\[
		\big\{ \rho_g(x) \mid x \in \langle n_1, \dots, n_k \rangle^\bullet \big\}
	\]
	has a limit point, then it must be $n_k/n_1$.
\end{lemma}

\begin{proof}
	Set $H = \langle n_1, \dots, n_k \rangle$. It is not hard to see that we can take $N \in \nn$ large enough such that for every $x \in H^\bullet$ there exists $r_x \in \ldb 1,N \rdb \cap H$ satisfying that $x = r_x + m_x n_1 n_k$ for some $m_x \in \nn$. Now fix $x_0 \in H$ with $x_0 > N$, and write $x_0 = r + m n_1 n_k$ for $r \in \ldb 1,N \rdb \cap H$ and $m \in \nn$. As $x_0 > N$, it follows that $m \ge 1$. Therefore any formal sum of copies of $n_1, \dots, n_k$ adding to $x_0$ and maximizing the number of distinguished generators (counting repetitions) must contain at least $m n_k$ copies of $n_1$ and so
	\begin{equation} \label{eq:max generalized length}
		\max \mathsf{L}_g(x_0) = \max \mathsf{L}_g(r + m n_1 n_k) = \max \mathsf{L}_g(r) + m n_k.
	\end{equation}
	Similarly, any formal sum of copies of $n_1, \dots, n_k$ adding to $x_0$ and minimizing the number of distinguished generators must contain at least $m n_1$ copies of $n_k$ and so
	\begin{equation} \label{eq:min generalized length}
		\min \mathsf{L}_g(x_0) = \min \mathsf{L}_g(r + m n_1 n_k) = \min \mathsf{L}_g(r) + m n_1.
	\end{equation}
	Using~(\ref{eq:max generalized length}) and~(\ref{eq:min generalized length}), we obtain that
	\[
		\rho_g(x_0)  = \frac{\max \mathsf{L}_g(r) + m n_k}{\min \mathsf{L}_g(r) + m n_1}.
	\]
	As a result,
	\[
		\big\{ \rho_g(x) \mid x \in H \setminus \ldb 1,N \rdb \big\} \subseteq \bigg\{ \frac{n_k + \frac 1m \max \mathsf{L}_g(r)}{ n_1 + \frac 1m \min \mathsf{L}_g(r)} \ \bigg{|} \ r \in \ldb 1,N \rdb \cap H \ \text{and} \ m \in \zz_{\ge 1} \bigg\}. 
	\]
	From the above inclusion of sets, it immediately follows that $\big\{\rho_g(x) \mid x \in H^\bullet \big\}$ can have at most one limit point, namely $n_k/n_1$.
\end{proof}
\medskip

\begin{remark}
	Lemma~\ref{lem:set of pseudo-elasticity} is essentially a generalization of~\cite[Corollary~2.3]{CHM06}, which states that the only limit point of the set of elasticities of a numerical monoid minimally generated by the elements $a_1 < a_2 < \dots < a_k$ (for $k \ge 2$) is $a_k/a_1$.
\end{remark}

\begin{remark}
	Another result similar to Lemma~\ref{lem:set of pseudo-elasticity} was previously established in~\cite[Corollary~4.5]{BOP17}. We decided to reprove it here not only for the sake of completeness, but also because we need to work in a more general context, meaning that our distinguished set of generators $\{n_1, \dots, n_k\}$ is not necessarily minimal, and $n_1, \dots, n_k$ are not necessarily relatively prime.
\end{remark}

For a nonzero vector $a \in \rr^d$, we let $\pp_{a} \colon \rr^d \to \rr a$ be the linear transformation that projects a vector of $\rr^d$ onto the one-dimensional space $\rr a$. Also, for each $j \in \ldb 1,d \rdb$, we let $p_j(x)$ denote the $j$-th component of $x$. In particular, for nonzero vectors $a, b \in \nn^2$, we have that
\[
	\frac{\langle a, b \rangle}{\norm{a}} = \frac{p_1(a) p_1(b) + p_2(a) p_2(b)}{ \norm{a}}
\]
is the Fourier coefficient of $b$ with respect to the unit vector $a/ \norm{a}$. Therefore the projection of $b$ on $a$ is given by
\begin{equation} \label{eq:Fourier coefficient}
\pp_a(b) = \frac{p_1(a) p_1(b) + p_2(a) p_2(b)}{\norm{a}^2} \, a.
\end{equation}

\begin{lemma} \label{lem:trigonometric identity}
	Let $a,x,y \in \nn^2$ such that $\slp(x) < \slp(a) < \slp(y)$. Also, let $\alpha$ be the acute angle between $x$ and $a$, and let $\beta$ be the acute angle between $a$ and $y$. Then the following identity holds:
	\[
		\big(\norm{a}  \norm{y} \sin \beta \big) x + \big(\norm{a} \norm{x} \sin \alpha \big) y = \norm{x} \norm{y} \sin (\alpha + \beta) a.
	\]
	Moreover, the coefficients of $x$, $y$, and $a$ in the above identity are nonnegative integers.
\end{lemma}

\begin{proof}
	Set $a^\perp := (-p_2(a), p_1(a))$, and note that
	\[
		\pp_{a^\perp}(x) = -\norm{x} \sin \alpha \frac{a^\perp}{\norm{a^\perp}} \ \ \text{and} \ \ \pp_{a^\perp}(y) = \norm{y} \sin \beta \frac{a^\perp}{ \norm{a^\perp}}.
	\] Taking $z = ( \norm{y} \sin \beta) x + ( \norm{x} \sin \alpha) y$, we obtain that $\pp_{a^\perp}(z) = 0$, which implies that~$z$ and $a$ are colinear, i.e., $\pp_a(z) = z$. Since
	\[
		\pp_a(x) = \norm{x} \cos \alpha \frac{a}{\norm{a}} \ \ \text{and} \ \  \pp_a(y) =  \norm{y} \cos \beta \frac{a}{\norm{a}},
	\]
	it follows that
	\begin{align*}
		z &= \pp_a(z) =  (\norm{y} \sin \beta) \pp_a(x) + ( \norm{x} \sin \alpha) \pp_a(y) \\
		&= \norm{x} \norm{y} \big(\sin \alpha \cos \beta + \sin \beta \cos \alpha) \frac{a}{\norm{a}} \\
		&= \norm{x} \norm{y}  \sin(\alpha + \beta) \frac{a}{ \norm{a} }. 
	\end{align*}
	Hence $\norm{a} z = \norm{x} \norm{y}  \sin(\alpha + \beta) a$, which is the desired trigonometric identity. Finally, observe that the coefficients of $a,x$, and $y$ represent areas of lattice parallelograms. Hence such coefficients must be nonnegative integers.
\end{proof}
\medskip

We are now in a position to prove that the elasticity of each monoid in $\mathcal{C}_2$ is either rational or infinite.

\begin{theorem} \label{thm:elasticity in dimension two}
	Let $H$ be a monoid in $\mathcal{C}_2$. Then $\rho(H)$ is either rational or infinite.
\end{theorem}

\begin{proof}
	If $H$ is finitely generated, then it follows by~\cite[Theorem~7]{AACS93} that $\rho(H)$ is rational. So we assume that $H$ is not finitely generated. Note that for every $v \in \nn^2$, the submonoid $\nn v \cap H$ of $H$ is isomorphic to an additive submonoid of $\nn$ and is, therefore, finitely generated. This, along with the fact that $|\mathcal{A}(H)| = \infty$, implies that the set $\slp(\mathcal{A}(H))$ must have at least one limit point (maybe $\infty$). By reflecting $H$ with respect to the line $y = x$ if necessary, we can assume that $\slp(\mathcal{A}(H))$ has a finite limit point.
	
	CASE 1. The set $\slp(\mathcal{A}(H))$ has at least two limit points. In this case, we will argue that $H$ has infinite elasticity. To do so, take $N \in \nn$.
	
	CASE 1.1. There exists $a \in \mathcal{A}(H)$ such that the vector $\slp(a)$ is strictly between two limit points of $\slp(H)$. Then the sets
	\[
		X := \{x \in \mathcal{A}(H) \mid \slp(x) < \slp(a)\}
	\]
	and
	\[
		Y := \{y \in \mathcal{A}(H) \mid \slp(y) > \slp(a)\}
	\]
	are both infinite. As a consequence, there exist atoms $x \in X$ and $y \in Y$ satisfying that $\min\{\norm{x}, \norm{y} \} \ge 2N \norm{a}$. By Lemma~\ref{lem:trigonometric identity}, one has that
	\begin{equation} \label{eq:coefficient of x,y,a}
		\big(\norm{a} \norm{y} \sin \beta \big) x + \big(\norm{a} \norm{x} \sin \alpha \big) y = \norm{x}  \norm{y} \sin (\alpha + \beta) a,	
	\end{equation}
	where $\alpha$ is the acute angle between $x$ and $a$, and $\beta$ is the acute angle between $a$ and~$y$. Using the fact that $\min\{\norm{x}, \norm{y} \} \ge 2N \norm{a}$, we obtain
	\begin{equation} \label{eq:coefficients for elasticity}
		\norm{x}  \norm{y} \sin(\alpha + \beta) > N( \norm{a}  \norm{y} \sin \beta + \norm{a}  \norm{x} \sin \alpha). 
	\end{equation}
	Because the coefficients of $x$, $y$, and $a$ in the identity~(\ref{eq:coefficient of x,y,a}) are positive integers, the element $h_0 := \norm{x} \norm{y} \sin (\alpha + \beta) a$ belongs to $H$. Moreover, applying the inequality~(\ref{eq:coefficients for elasticity}), one obtains that
	\[
		\rho(H) \ge \rho(h_0) \ge \frac{ \norm{x}  \norm{y} \sin (\alpha + \beta)  } { \norm{a}  \norm{y} \sin \beta + \norm{a}  \norm{x} \sin \alpha }\ge N.
	\]
	Hence $\rho(H) = \infty$.
	
	CASE 1.2. There is no $a \in \mathcal{A}(H)$ such that $\slp(a)$ is strictly between two limit points. Observe that, in this case, $\slp(H)$ contains exactly two limit points. As a consequence, it is not hard to see that we can choose $x,y,a \in \mathcal{A}(H)$ such that $\slp(x) < \slp(a) < \slp(y)$. In addition, note that we can assume that $\norm{a}$ is large enough that the inequalities
	\[
		\norm{a} \sin \beta > N \norm{x}/2 \quad \text{and} \quad \norm{a} \sin \alpha > N \norm{y}/2
	\]
	hold, where the angles $\alpha$ and $\beta$ are defined as in the CASE~1.1. By Lemma~\ref{lem:trigonometric identity}, we again obtain the identity~(\ref{eq:coefficient of x,y,a}). Because the coefficients of $x$, $y$, and $a$ in~(\ref{eq:coefficient of x,y,a}) are positive integers, $h_1 = 2 \norm{x}  \norm{y} \sin (\alpha + \beta) a$ belongs to $H$. Using the inequalities $\norm{a} \sin \beta > N \norm{x}/2$ and $\norm{a} \sin \alpha > N \norm{y}/2$, we get
	\[
		\rho(H) \ge \rho(h_1) \ge \frac { \norm{a} \norm{y} \sin \beta + \norm{a} \norm{x} \sin \alpha } { \norm{x} \norm{y} \sin (\alpha + \beta)  } > N.
	\]
	Thus, in this case, $\rho(H) = \infty$.
	
	CASE 2. The set $\slp(\mathcal{A}(H))$ contains only one limit point. Let $\ell$ be the limit point of $\slp(\mathcal{A}(H))$. Now consider the set
	\[
		X_\ell := \{x \in \mathcal{A}(H) \mid \slp(x) < \ell\}
	\]
	and the set
	\[
		  \ Y_\ell := \{y \in \mathcal{A}(H) \mid \slp(y) > \ell\}.
	\]
	
	CASE 2.1. The sets $X_\ell$ and $Y_\ell$ are both nonempty. Fix $N \in \nn$. Take $x \in X_\ell$ and $y \in Y_\ell$.  Since $\ell$ is a limit point of $\slp(H)$, we can choose $a \in \mathcal{A}(H)$ satisfying that $\slp(x) < \slp(a) < \slp(y)$ and large enough such that the inequalities
	\[
		\norm{a} \sin \beta > N \norm{x}/2 \quad \text{and} \quad \norm{a} \sin \alpha > N \norm{y}/2
	\]
	hold. Proceeding exactly as we did in CASE 1.2, we can conclude that $\rho(H) > N$. Hence $\rho(H) = \infty$ in this case again.
	
	CASE 2.2. One of the sets $X_\ell$ and $Y_\ell$ is empty. Assume, without loss of generality, that $X_\ell$ is not empty. Take a nonzero vector $v \in \rr^2_{\ge 0}$ such that $\slp(v) = \ell$, and set
	\[
		v^\perp = \frac{(-p_2(v), p_1(v))}{\norm{(-p_2(v), p_1(v))}}.
	\]
	Now define the set
	\[
		S_\ell = \{\norm{\pp_{v^\perp}(a)} \mid a \in \mathcal{A}(H)\}.
	\]
	
	CASE 2.2.1. The set $S_\ell$ is not finite. Fix $x \in X_\ell$. As $S_\ell$ is infinite and the Fourier coefficient of each vector in $X_\ell$ with respect to the normal vector $v^\perp$ is an integer, there exists $a \in X_\ell$ such that $\norm{\pp_{v^\perp}(a)} > 2N \norm{x}$. Note that $\norm{\pp_{v^\perp}(a)} = \norm{a} \sin \beta'$, where $\beta'$ is the acute angle between $a$ and $v$. Now take $y \in X_\ell$ such that the acute angle $\beta$ between $y$ and $a$ is close enough to $\beta'$ that both inequalities $\slp(y) > \slp(a)$ and $\norm{a} \sin \beta > N \norm{x}$ hold. Now, we can apply Lemma~\ref{lem:trigonometric identity} to obtain once again the identity~(\ref{eq:coefficient of x,y,a}). Since the coefficients of $x$, $y$, and $a$ in~(\ref{eq:coefficient of x,y,a}) are positive integers, the element $h_2 := 2 \norm{x}  \norm{y} \sin (\alpha + \beta) a$ belongs to $H$. On the other hand, using the fact that $\norm{a} \sin \beta> N \norm{x}$, one finds that
	 \[
	 	\rho(H) \ge \rho(h_2) \ge \frac { \norm{a} \norm{y} \sin \beta + \norm{a} \norm{x} \sin \alpha } { \norm{x} \norm{y} \sin (\alpha + \beta)  } > \frac { \norm{a} \sin \beta } { \norm{x} \sin (\alpha + \beta)  } > N.
	 \]
	 This allows us to conclude again that $\rho(H) = \infty$.
	 
	 CASE 2.2.2. The set $S_\ell$ is finite. Take $H_\ell = \langle s_1, \dots, s_k \rangle$, where $S_\ell = \{s_1, \dots, s_k\}$ and $s_1 < \dots < s_k$. Among all the atoms of $H$ minimizing the set $S_\ell$, let $x$ be the one of largest slope. On the other hand, among all the atoms of $H$ maximizing $S_\ell$, let $a$ be the one with largest slope. Fix $\epsilon > 0$.
	 
	 First, suppose that $\slp(a) < \slp(x)$, and let $\alpha$ be the acute angle between $x$ and~$a$. Now take $y \in \mathcal{A}(H)$, and let $\beta$ denote the acute angle between $a$ and $y$. Since
	 \[
	 	\lim_{\norm{y} \to \infty} \frac{\norm{a} \sin \beta}{ \norm{x} \sin(\alpha + \beta)} = \frac{ \norm{\pp_{v^\perp}(a)} }{ \norm{\pp_{v^\perp}(x)}} = \frac{s_k}{s_1},
	 \]
	 we can assume that $\norm{y}$ is large enough such that the inequalities $\slp(a) > \slp(y)$ and
	 \begin{equation} \label{eq:epsilon inequality 1}
	 	\frac{ \norm{a} \sin \beta}{ \norm{x} \sin(\alpha + \beta)} > \frac{s_k}{s_1} - \epsilon
	 \end{equation}
	 both hold. Now, Lemma~\ref{lem:trigonometric identity} allows us to use identity~(\ref{eq:coefficient of x,y,a}) once again. This, along with the fact that $h_3 := \norm{x} \norm{y} \sin(\alpha + \beta)a \in H$, implies that
	 \[
	 	\rho(H) \ge \rho(h_3) \ge \frac{\norm{a} \sin \beta + \big(\norm{x} / \norm{y} \big) \norm{a} \sin \alpha}{ \norm{x} \sin (\alpha + \beta)} > \frac{s_k}{s_1} - \epsilon .
	 \]
	 Thus, $\rho(H) \ge s_k/s_1 \ge \rho_g(H_\ell)$.
	 
	 Now suppose that $\slp(a) > \slp(x)$. Let $\alpha$ be defined as before, take $y \in \mathcal{A}(H)$, and let $\beta$ now denote the acute angle between $x$ and $y$. Since
	 \[
		 \lim_{ \norm{y} \to \infty} \frac{ \norm{a} \sin(\alpha + \beta)}{\big(\norm{x} / \norm{y}\big) \norm{a} \sin \alpha + \norm{x} \sin \beta} = \frac{ \norm{\pp_{v^\perp}(a)}}{ \norm{\pp_{v^\perp}(x)}} = \frac{s_k}{s_1}
	 \]
	 and $\ell$ is a limit point of $\slp(H)$, we can assume that $\norm{y}$ is large enough so that $\slp(x) > \slp(y)$ and
	 \[
	 	\frac{\norm{a} \sin(\alpha + \beta)}{\big(\norm{x} / \norm{y}\big) \norm{a} \sin \alpha + \norm{x} \sin \beta} > \frac{s_k}{s_1} - \epsilon.
	 \]
	 By Lemma~\ref{lem:trigonometric identity},
	 \begin{equation} \label{eq:coefficient of a,x,y,}
	 	\big(\norm{x} \norm{y} \sin \beta \big) a + \big(\norm{x} \norm{a} \sin \alpha \big) y = \norm{a} \norm{y} \sin (\alpha + \beta) x.
	 \end{equation}
	 Since the coefficients in~(\ref{eq:coefficient of a,x,y,}) are positive integers, $h_4 := \norm{a} \norm{y} \sin(\alpha + \beta)x$ is an element of $H$ and, therefore,
	 \[
	 	\rho(H) \ge \rho(h_4) \ge \frac{\norm{a} \sin (\alpha + \beta)}{\big(\norm{x}/ \norm{y} \big) \norm{a} \sin \alpha + \norm{x} \sin \beta } \ge \frac{s_k}{s_1} - \epsilon.
	 \]
	 As a consequence, $\rho(H) \ge s_k/s_1 \ge \rho_g(H_\ell)$.
	 
	 Finally, suppose that $\slp(a) = \slp(x)$. In this case, it is not hard to see that the element $h_5 := p_1(a) p_2(a) x \in H$ can also be written in the form $h_5 = p_1(a)p_2(x) a$. Since $a,x \in \mathcal{A}(H)$, it follows that
	 \[
	 	\rho(H) = \rho(h_5) \ge \frac{p_2(a)}{p_2(x)} = \frac{\norm{a}}{\norm{x}} = \frac{s_k}{s_1} \ge \rho_g(H_\ell).
	 \]
	 Thus, we always have $\rho(H) \ge \rho_g(H_\ell)$. On the other hand, if $a_1 + \dots + a_n \in \mathsf{Z}_H(w)$ is an $n$-length factorization of $w \in H$, then $\norm{\pp_{v^\perp}(a_1)} + \dots + \norm{\pp_{v^\perp}(a_n)}$ is an $n$-length generalized factorization of $\norm{\pp_{v^\perp}(w)}$ in $H_\ell$. Therefore $\mathsf{L}_H(w) \subseteq \mathsf{L}_{g(H_\ell)}(\norm{\pp_{v^\perp}(w)})$ for all $w \in H$, where $\mathsf{L}_{g(H_\ell)}(h)$ denotes the generalized set of lengths of $h$ in $H_\ell$ with respect to the distinguished set of generators $s_1, \dots, s_k$. This implies that $\rho(H) \le \rho_g(H_\ell)$. Hence $\rho(H) = \rho_g(H_\ell)$, which is rational by Lemma~\ref{lem:set of pseudo-elasticity}. This completes the proof.
\end{proof}
\medskip

We conclude our exposition providing a subclass of $\mathcal{C}_d$ (for $d \ge 3$) whose members have rational or infinite elasticity.

\begin{theorem} \label{thm:elasticity of polyhedral GAMS}
	Let $H$ be a monoid in $\mathcal{C}_d$ with $d \ge 3$. If the cone of $H$ in the $\qq$-space $\qq \otimes_{\zz} \emph{gp}(H)$ is polyhedral, then $\rho(H)$ is either rational or infinite.
\end{theorem}

\begin{proof}
	Set $d = \rank(H)$. After identifying $\qq \otimes_{\zz} \text{gp}(H)$ with $\qq^d$, we can assume that $H$ is a submonoid of $\nn^d$. If $H$ is finitely generated, then we can argue that $\rho(H) \in \qq$ as we did at the beginning of the proof of Theorem~\ref{thm:elasticity in dimension two}. Then there is no loss in assuming that $H$ is not finitely generated, i.e., $|\mathcal{A}(H)| = \infty$. \\
	
	Fix $N \in \nn$. Since $\cone(H)$ is polyhedral, it must have finitely many one-dimensional faces; call them $L_1, \dots, L_n$. Since each $L_i$ is a one-dimensional face, we can take $a_i \in L_i \cap \mathcal{A}(H)$ for each $i \in \ldb 1,n \rdb$. Clearly, $\cone(H) = \cone(a_1, \dots, a_n)$. Consider the parallelepiped
	\[
		\Pi := \big\{\alpha_1 a_1 + \dots + \alpha_n a_n \mid 0 \le \alpha_i \le 1 \ \text{for every} \ i \in \ldb 1, n \rdb \big\}.
	\]
	Since $\Pi \cap \zz^d$ is finite and
	\[
		\Pi \cap \zz^d \subset \qq_{\ge 0} a_1 + \dots + \qq_{\ge 0} a_n,
	\]
	we can choose $N_0 \in \nn^\bullet$ large enough that $N_0 z \in \zz_{\ge 0} a_1 + \dots + \zz_{\ge 0} a_n$ for each $z \in \Pi \cap \zz^d$. Notice that for each $n \in \nn$, there exist only finitely many atoms of $H$ whose norms are at most $n$. 
	As a result, there exists $a \in \mathcal{A}(H)$ such that
	\[
		\norm{a} > (N + 1)(\norm{a_1} + \dots + \norm{a_n}).
	\]
	
	As we can naturally partition $\cone(H)$ into copies of the parallelepiped $\Pi$, we can write $a = v + c_1 a_1 + \dots + c_n a_n$ for some $v \in \Pi \cap \zz^d$ and nonnegative integer coefficients~$c_1, \dots, c_n$. Since the diameter of $\Pi$ is $\norm{a_1 + \dots + a_n}$ and $v \in \Pi$, it follows that $\norm{v} \le \norm{a_1} + \dots + \norm{a_n}$ and, therefore,
	\[
		1 + \sum_{i=1}^n c_i \ge \frac{\norm{v}}{\sum_{i=1}^n \norm{a_i}} + \frac{\norm{\sum_{i=1}^n c_ia_i}}{\sum_{i=1}^n \norm{a_i}} \ge \frac{\norm{a}}{\sum_{i=1}^n \norm{a_i}} > 1 + N.
	\]
	Hence $c_1 + \dots + c_n > N$. Taking $c'_1, \dots, c'_n \in \zz_{\ge 0}$ such that $N_0 v = c'_1 a_1 + \dots + c'_n a_n$, we obtain that
	\[
		N_0 a = N_0 v + \sum_{i=1}^n N_0c_i a_i = \sum_{i=1}^n (c'_i + N_0 c_i) a_i.
	\]
	Therefore
	\[
		\rho(H) \ge \rho(N_0 a) \ge \frac{\sum_{i=1}^n (c'_i + N_0 c_i)}{N_0} \ge \sum_{i=1}^n c_i > N.
	\]
	As $N$ was arbitrarily taken, $\rho(H) = \infty$, which concludes the proof.
\end{proof}
\medskip
	
Recall that an atomic monoid $H$ is fully elastic if $\mathcal{R}(H) = \{q \in \qq \mid 1 \le q \le \rho(H)\}$. Each monoid in $\mathcal{C}_1$ fails to be fully elastic (see~\cite[Theorem~2.2]{CHM06}). However, because every atomic monoid having full system of sets of lengths is, obviously, fully elastic, we have the following direct implication of Corollary~\ref{cor:GAMS with full system of sets of lengths}.
	
\begin{prop}
	For each $d \ge 2$, there exists a monoid in $\mathcal{C}_d$ that is fully elastic.
\end{prop}
\bigskip

\section*{Acknowledgements}

	While working on this paper, the author was supported by the NSF-AGEP Fellowship and by the UC Year Dissertation Fellowship. The author would like to thank Alfred Geroldinger for helpful suggestions on early versions of this paper. In addition, the author would like to thank Winfried Bruns for providing the main idea of the proof of Lemma~\ref{lem:basic isomorphic representation} and Salvatore Tringali for proposing the question motivating Section~\ref{sec:BF monoid with full elasticity}. Finally, the author is grateful to an anonymous referee, whose careful revision and comments help improve the final version of this paper.
	
	\bigskip
	

\end{document}